\documentclass[11pt,amsfonts, epsfig]{amsart}
\usepackage{amsmath, amscd, amssymb}

\usepackage{stmaryrd}
\usepackage{graphpap, color}
\usepackage{mathrsfs}
\usepackage{pstricks}
\usepackage{color}
\usepackage{cancel}
\usepackage[mathscr]{eucal}
\usepackage[notcite, notref]{showkeys}
\usepackage{pstricks}
\usepackage{color}
\usepackage{cancel}
\usepackage{verbatim}
\usepackage{latexsym}
\usepackage[all]{xy}

\def\bE{{\mathbf E}}

\def\bC{{\mathbf C}}

\def\ff{{\mathfrak f}}

\def\Gm{{\mathbb{G}_m}}
\newcommand{\bG}{ \mathbf{G} }
\def\qme{{\barM^{\epsilon}_{g,\ell}}}

\numberwithin{equation}{section}

\def\sO{{\mathscr O}}

\def\sL{{\mathscr L}}

\def\sO{\mathscr{O}}

\def\sE{\mathscr{E}}

\def\sR{\mathscr{R}}

\newcommand{\CC}{\mathbb{C}}
\newcommand{\EE}{\mathbb{E}}
\newcommand{\LL}{\mathbb{L}}
\newcommand{\PP}{\mathbb{P}}
\newcommand{\QQ}{\mathbb{Q}}

\newcommand{\ZZ}{\mathbb{Z}}
\newcommand{\GG}{\mathbb{G}}

\newcommand{\fP}{\mathbb{M}}

\newcommand{\TT}{\mathbb{T}}
\newcommand{\DD}{\mathbb{D}}

\newcommand{\bx}{\mathbf{x}}

\newcommand{\cal}{\mathcal}

\def\ee{\mathfrak e}
\def\cB{{\cal B}}
\def\cC{{\cal C}}
\def\cD{{\cal D}}
\def\cE{{\cal E}}

\def\cH{{\cal H}}
\def\cK{{\cal K}}
\def\cL{{\cal L}}

\def\cN{{\cal N}}
\def\cO{{\cal O}}
\def\cP{{\cal P}}
\def\cQ{{\cal Q}}

\def\cU{{\cal U}}
\def\cV{{\cal V}}

\def\cX{{\cal X}}

\def\cZ{{\cal Z}}

\def\fC{\mathfrak{C}}

\def\fX{\mathfrak{X}}
\def\fY{\mathfrak{Y}}

\def\ff{\mathfrak{f}}

\def\sP{{\mathscr P}}

\def\mapright#1{\,\smash{\mathop{\lra}\limits^{#1}}\,}

\def\dual{^{\vee}}

\def\sta{^\ast}

\def\virt{^{\mathrm{vir}}}

\def\sta{^{\ast}}

\def\sta{^*}

\def\lra{\longrightarrow}

\def\lsta{_{\ast}}

\newcommand{\ep}{\epsilon}
\newcommand{\lam}{\lambda}

\def\begeq{\begin{equation}}
\def\endeq{\end{equation}}
\def\and{\quad{\rm and}\quad}
\def\bl{\bigl(}
\def\br{\bigr)}
\def\defeq{:=}

\def\sub{\subset}
\def\Ao{{\mathbb A}^{\!1}}

\def\and{\quad\text{and}\quad}

\DeclareMathOperator{\pr}{pr}

 \DeclareMathOperator{\rank}{rank}

\newtheorem{prop}{Proposition}[section]

\newtheorem{theo}[prop]{Theorem}
\newtheorem{lemm}[prop]{Lemma}
\newtheorem{coro}[prop]{Corollary}

\newtheorem{defi}[prop]{Definition}

\newtheorem{defi-prop}[prop]{Definition-Proposition}

\DeclareMathOperator{\coker}{coker}

\def\Ob{\cO b}

\def\bul{^\bullet}

\def\Pn{{\mathbb P}^n}

\def\sta{^\ast}

\let\lab=\label

\def\sO{{\mathscr O}}

\def\sR{{\mathscr R}}

\def\beq{\begin{equation}}
\def\eeq{\end{equation}}

\def\Pf{{\PP^4}}

\def\fP{{\mathfrak P}}

\def\fZ{{\mathfrak Z}}

\def\bee{\begin{equation}}
\def\eeq{\end{equation}}

\def\sC{{\mathscr C}}

\def\fN{{\mathfrak N}}

\def\ti{\tilde}

\def\fg{{\mathfrak g}}
\def\barM{{\overline{M}}}

\def\fq{\mathfrak q}

\def\cpdp{\cP} 

\def\cqg{{\cQ_{g}}}

\def\cng{{\cN_{g}}}
\def\cngp{\cN} 
\def\Vb{\mathrm{Vb}}

\def\cpg{\cP}
\def\cpgp{\cP}

\def\cvgp{\cV} 

\def\AA{\mathbb A}
\def\lAo{_{\Ao}}
\let\lab=\label

\title{Invariants of
stable quasimaps with fields}

\author{Huai-Liang Chang}
\author{Mu-lin Li}
\address{Mathematics Department, Hong Kong University of Science and Technology}
\address{College of Mathematics and Econometrics, Hunan University}
\thanks{H.L. Chang was partially supported by HK GRF grant 16301515 and 16301717;  M.L. Li was partially supported by the Start-up Fund of Hunan University.}

\begin{document}
\maketitle

\begin{abstract} For arbitrary smooth hypersurface $X\subset\PP^n$
we construct moduli of   quasimaps with P fields.
Apply Kiem-Li's cosection localization
we obtain a virtual fundamental class. We show the class coincides, up to sign,
with that of moduli of   quasimaps to $X$. This
generalizes Chang-Li's \cite{CL} numerical identity to the cycle level, and from Gromov Witten invariants to quasimap invariants.
\end{abstract}
\section{introduction}
 

  Recently there are more approaches to GW invariants using Landau Ginzburg theory. One example is Jun Li and the first author's work \cite{CL} identifying quintics' GW invariants with its GW invariants of maps with fields, which can be regarded as enumerating curves ``in" the LG space $(K_\Pf,W)$, where $W$ is a regular funciton on $K_\Pf$ induced by quintic polynomials.
Such point of view enables some applications: (i)  providing  an algebraic proof of  Li-Zinger relation  \cite{CL1} so that $g=1$ GW invariants can be approached by reduced  $g=1$ GW invariants (recover \cite{LZ} algebraically); (ii)  an all genus wall crossing from CY's GW invariants to FJRW invariants can be realized by the moduli of Mixed-Spin-P fields (\cite{CLLL,CLLL1})\footnote{this also gives another way to obtain Zinger's genus one formula \cite{Zi} of quintic GW;}.

  Recently, Ionut Ciocan-Fontanine, Bumsig Kim \cite{FK1} used stable quasimaps to give a mathematical description of $B$-side(also see Clader-Janda-Ruan \cite{CJR}).  They use all genus wall crossing to provide relations between GW and quasimap invariants. The relations is interpreted as mirror maps.
  Then using that $\ep=0^+$-quasimap moduli admits bundle's euler class description, B. Kim and H. Lho \cite{K-L}   also obtain quintics $g=1$ GW invariants recovering Zinger's formula. This generalized stability is also used by Fan, Jarvis, Ruan  in Gauged Linear Sigma Model \cite{FJR}.


  In this paper  we consider  the moduli of $\epsilon$-stable quasimaps to $\PP^n$ with fields(also c.f. \cite{FJR}). Similar to \cite{CL} in many aspects, the moduli spaces are   not proper for positive genus, and admit virtual cycles by  using Kiem-Li's cosection localization. Our main theorem is that these virtual cycles coincide (up to signs) with the virtual cycles of the $\epsilon$-stable quasimaps moduli of the smooth hypersurface $X=(q(x)=0)\subset \PP^n$ with $\deg q=m$.   This generalizes the result in \cite{CL} to arbitrary hypersurfaces, with markings, and lifts  the identity between invariants (\cite[Thm.\, 1.1]{CL}) to identity between virtual cycles. \black

We briefly outline our construction and the theorem.  Let $q=q(x)$ be a degree $m$ homogenus polynomial. Assume the hypersurface   $X:=(q=0)\subset\PP^n$  is smooth. Let $\qme(\PP^n,d)$ be the moduli of genus $g$ degree $d$ $\epsilon$-stable quasimaps to $\PP^n$. Every closed point of $\qme(\PP^n,d)$ is of the form
$$\xi=( C, p_1,\cdots p_\ell\in C, L\in Pic(C), u\in \Gamma(L^{\oplus n+1})).$$
We let $\qme(\PP^n,d)^p$ be the moduli of $(\xi,p)$ where $p\in\Gamma(L^{\vee\otimes m}\otimes\omega_{C})$.
  It forms a Deligne-Mumford stack. 
When $g$ is positive, it is not proper. Parallel to \cite{CL},  the homogeneous polynomial $q(x)$ determines a cosection  of its obstruction sheaf:  $$\sigma : \Ob_{\qme(\PP^n,d)^p}\lra \sO_{\qme(\PP^n,d)^p}.
$$
 The non-surjective loci (the degeneracy loci) of the cosection is
$$
\qme(X,d)\sub \qme(\PP^n,d)^p,  $$
and is proper. The cosection localization of Kiem-Li  induces a  virtual cycle
$$[\qme(\PP^n,d)^p]\virt\in A_{\delta} \qme(X,d).
$$
\begin{theo}\label{theo1}
For $g\ge0$, $\epsilon>0$ and $\ell$ be a nonnegative integer, then we have
$$
[\qme(\PP^n,d)^p]\virt=(-1)^{m d+1-g}[\qme(X,d)]\virt\in A_{\delta} \qme(X,d).
$$
\end{theo}

 In  case  $X$ is the quintic CY hypersurface and $\ell=0$, we have virtual dimension $\delta=0$. Taking degree of the identity   shows that
 the quasimap invariant of quintics $\deg [\qme(X,d)]\virt$ admits a Landau-Ginzburg alternate construction $\deg [\qme(\PP^n,d)^p]\virt$.
 When $\ep=\infty$, this recovers the main theorem in \cite{CL}.

{\bf Notation }
  Every stack in this paper is over $\CC$. Fixing  $g$ and $\ell$ in whole paper, we use  $\cD:=\mathcal{D}_{g,\ell}$ to denote the smooth Artin stack of $\ell$-pointed nodal curve with a line bundle (c.f. Definition \ref{cD}). The symbol  $\TT$ denotes the (relevent) tangent complex, which is the derived dual of cotangent complex $\LL$.

{\bf Acknowledgement}  The authors are greatful to Jun Li, Young Hoon Kiem, Ionut Ciocan-Fontanine and Yongbin Ruan
for helpful discussions.  The authors learn that B. Kim and J. OH have obtained the same result using localized chern characters, a method very different from ours. \black

\section{deformation theory of Quasimaps  via direct image cones}




\subsection{Quasimaps to $\PP^n$ and smooth hypersurface}
 Fix a smooth hypersurface   $X=(q=0)$ of $\Pn$ where $q$ is a degree $m$ homogeneous polynomial. We briefly recall the $\epsilon$-stable quasimaps (defined in \cite[Definition 3.1.1]{FK0}) to $\Pn$ and to $X$. 

\begin{defi} \label{stable qmap}   A   prestable   quasimap  to $\PP^n$    consists of
$$( C,p_1,\dots ,p_\ell , L, \{ u_i\}_{i=1}^{n+1}),$$
where
\begin{itemize}
\item $(C,p_1,\dots ,p_\ell )$ is a connected nodal curve of genus $g$ with $\ell$ markings,
\item $L$ is a degree-$d$ line bundle over $C$ with $u_1,\cdots,u_{n+1} \in \Gamma (C, L )$,
\item (nondegeneracy) the set $B\subset C$ of zeros of $(u_1,\cdots,u_{n+1})$
is finite and disjoint from the nodes and markings on $C$.


\end{itemize}

\end{defi}


  Denote $u=(u_1,\cdots,u_{n+1}):\sO_C^{\oplus (n+1)}\to L$.

\begin{defi}The length at $z\in C$ of a quasimap $( C,p_1,\cdots,p_\ell ,  L, \{ u_i\}_{i=1}^{n+1})$ is
$$
leng(z):=\text{length}_{z}(\text{coker}(\sO_C^{\oplus (n+1)}\to L)).
$$
\end{defi}
\black



 Let $\ep$ be a positive rational number.

\begin{defi}\label{stability}  A prestable quasimap $( C,p_1,\dots ,p_\ell ,  L, \{ u_i\}_{i=1}^{n+1})$ is $\epsilon$-stable if
\begin{enumerate}

\item $\omega _C(p_1+\dots +p_\ell ) \otimes  L^\epsilon $ is ample.
\item $\epsilon \, leng   (z)\le 1$ for every point $z$ in $C$.

\end{enumerate}
\end{defi}
 As in \cite{FK} we denote by $\epsilon=0^{+}$ when $\epsilon\to 0$. 


\begin{defi}\label{family} A family of genus $g$ degree $d$ $\epsilon$-stable quasimaps to $\PP^n$ over a   scheme $S$ consists of the data
$$( \pi:\cC_S \to S, \{p_i: S \rightarrow \cC_S\}_{i=1,\dots,\ell}, \sL_S, \{u_j\}_{j=1}^{n+1})$$
where
\begin{itemize}
\item $\pi:\cC_S \rightarrow S$ is a flat family of connected nodal curves over $S$,
\item  $p_1,\cdots,p_\ell $ are  sections of $\pi$,
\item $\sL_S$ is a line bundle over $\cC_S$ with  degree $d$ along fibers of $\cC_S/S$,
\item $u_j\in\Gamma(\cC_S,\sL_S)$ for $i=1,\cdots,n+1$,
\end{itemize}
such that the restriction to every geometric fiber $\cC_s$ of $\pi$ is an $\epsilon$-stable $\ell$-marked quasimap of genus $g,\epsilon$ and $d$.
 An arrow from $(\cC_S / S, \cdots)$ to $(\cC'_{S} / S, \cdots)$
are $S$-isomorphisms  $f:\cC_S \rightarrow \cC'_{S}$  and
$\sigma_f: \sL_S \rightarrow f^*\sL'_{S}$ preserving markings and sections.
\end{defi}
 Given a scheme $S$, denote $\qme(\PP^n,d)(S)$ to be the set of all   $(\cC_S/S,\{p_i\}, \sL_S, \{u_j\})$.  Then, with naturally defined arrows,  $\qme(\PP^n,d)$ forms a groupoid fibered over the category of schemes.
 By \cite[Theorem 4.0.1]{FK0}, or \cite[Example 3.1.4]{FKM} and \cite[Thm.\, 7.1.6]{FKM}, we have  \black 

\begin{lemm}\label{equi} The groupoid $\qme(\PP^n,d)$ is  a DM stack, proper over $\CC$. 
\end{lemm}
 We recall the notion of the $\epsilon$-stable quasimaps  to the hypersurface $X$ (of $\Pn$).  
 \footnote{when $\epsilon=0^{+}$ the definition is exactly the stable quotient to $X$ defined in \cite[Sect. 10.1]{MOP}.}

\begin{defi}\label{family1} A family of genus $g$ degree $d$ $\epsilon$-stable quasimaps to $X$ over a base scheme $S$ consists of the data
$$(\pi: \cC_S \rightarrow S, \{p_i: S \rightarrow \cC_S\}_{i=1}^\ell, \sL_S, \{u_i\}_{i=1}^{n+1})\in \qme(\PP^n,d)(S)$$
where  the section $q(u_i)\in\Gamma(\cC_S,\sL_S^{\otimes m})$ vanishes.

 Arrows are defined the same as Definition \ref{family}.
\end{defi}

 Let  $\qme(X,d)(S)$   be the set of all  families in above definition. With naturally defined arrows $\qme(X,d)$ forms a closed substack of $\qme(\PP^n,d)$. 


 \subsection{Deformation theory}
\subsubsection{Deformation theory via the direct image cone}

  We recall the direct image cone construction in \cite{CL}. Let $\pi:\cC\to \fY$ be a flat family of connected nodal genus g curves, where $\fY$ is a  smooth  Artin stack. Let $\cZ  $ be an Artin stack
representable and quasi-projective over $\cC$. Define a groupoid $\fX  $ as follows. For any scheme $S\to\fY$,   denote
$\cC_S=\cC\times_\fY S$ and $\cZ_S=\cZ\times_\cC  \cC_S$. Then $\cZ\to\cC$ induces the projection $\pi_S:\cZ_S\to \cC_S$. Let
$$\fX{}(S)=\{ s:\cC_S\to \cZ_S\mid s \text{ are $\cC_S$-morphisms}\,\}.
$$
The arrows are defined by pullbacks. \black

\begin{prop}\label{ZA}\cite[Prop.\, 2.3]{CL}
The groupoid $\fX{}$ is an Artin stack with a natural representable and quasi-projective  morphism to $\fY$.
\end{prop}
 
In case $\cZ=\text{Vb}(\sE)$  is the underling vector bundle of some locally free sheaf $\sE$ over $\cD$, 
the space $C(\pi_*\sE)$ is the moduli of sections, denoted as $C(\pi_*\sE)\cong \fX{}$ in \cite[Coro.\, 2.4]{CL}
 
 Let $\pi_\fX:\cC_\fX \to \fX$
 be the universal family of $\fX{}$ and let $\ee:\cC_\fX\to \cZ$ be the tautological
evaluation map. 
  We generalize \cite[Prop.\, 2.5]{CL}  to the following.
\begin{prop}\label{deformation} Suppose $\cZ\to\cC$ is a flat morphism. Then $\fX{}\to\fY$ has a relative obstruction theory
\begin{align*}\phi_{\fX{}/\fY}: \TT_{\fX{}/\fY}\lra \EE_{\fX{}/\fY}\defeq R^\bullet \pi_{\fX\ast} \ee\sta \TT_{\cZ/\cC}.
\end{align*}
\end{prop}
\begin{proof}
 In the case $\cZ/\cC$ is smooth, the claim is proved in \cite[Prop.\, 2.5, Appendix A.3]{CL}.
All the arguments applies to the case $\cZ/\cC$ is flat. For example, the equality in the second lines under \cite[(A.9)]{CL} holds by the flat base change property of cotangent complex (\cite[Coro.\, 2.2.3]{Illusie}).

\end{proof}
\subsubsection{Perfect obstruction theory of the quasimap moduli}
\begin{defi}\label{cD}
Let $\cD:=\mathcal{D}_{g,\ell}$ be the groupoid associating to each scheme S the set $\cD(S)=\{(\mathcal{C}_S,\{p_i:S\to\cC_S\}_{i=1}^\ell,\cL_S)\}$, where $\pi:\mathcal{C}_S \rightarrow S$ is a flat family of connected nodal curves
and $\cL_S$ is a line bundle on $\mathcal{C}_S$.  An arrow from
$(\mathcal{C}_S,\{p_i:S\to\cC_S\}_{i=1}^\ell,\sL_S)$ to $ (\mathcal{C}^{\prime}_S,\{p^{\prime}_i:S\to\cC_S^{\prime}\}_{i=1}^\ell,\cL'_S)$ consists of $f:\mathcal{C}_S\rightarrow\mathcal{C}^{\prime}_S$ and an isomorphism $\tau:\cL_S\rightarrow f^{*}\cL_S^{\prime}$, which preserve the markings.
\end{defi}


 Let $\cC_{\cD}\mapright{\pi_\cD}\cD$ and $\sL_\cD$ be the universal curve and line bundle of $\cD$, and set $\cZ:=\mbox{Vb}( \mathscr{L}_{\cD}^{\oplus (n+1)})$ over $\cC_{\cD}$. There is an arrow
 $$\lam:\qme(\PP^n,d) \lra \fX\cong C(\pi_{\cD\ast} \cL_\cD^{\oplus n+1}),$$
 sending the data in Definition \ref{family} to the associated $S\to \cD$ with sections given by $\{u_j\}'s$. 
 Since the nondegeneracy condition in Definition \ref{stable qmap} and stability conditions in Definition \ref{stability} are open conditions, $\lam$ is an open embedding. 

 By Proposition \ref{ZA} and \ref{deformation}.
 $\qme(\PP^n,d)/\cD$ has an obstruction theory
$$\phi_{\qme(\PP^n,d)/\cD}:\TT_{\qme(\PP^n,d)/\cD}\to \EE_{\qme(\PP^n,d)/\cD}:=R^\bullet \pi_{\fX\ast} \ee\sta \TT_{\cZ/\cC_{\cD}}.$$
  Since $\cZ$ is smooth over $\cC_{\cD}$,  we know $R^\bullet \pi_{\fX\ast} \ee\sta \TT_{\cZ/\cC_{\cD}}$ has amplitude contained in $[0,1]$ and   $\phi_{\qme(\PP^n,d)/\cD}$   is a relative perfect obstruction theory.

Let $\fP$ be the universal $\GG_m$ principal bundle over $\cC_{\cD}$, then $\cZ\cong\fP\times_{\GG_m}\CC^{n+1}$.  We have the following canonical map
$$
 \Psi:\cZ=\mbox{Vb}( \mathscr{L}_{\cD}^{\oplus (n+1)})\lra [\CC^{n+1}/\GG_m],
$$

  Let $C_X$ be the cone $(q(x_i)=0)\subset\CC^{ {n+1}}$. Denote
\beq\label{cZX}
\cZ_X:=\fP\times_{\GG_m}C_{X}\cong\mbox{Vb}( \mathscr{L}_{\cD}^{\oplus (n+1)})\times_{[\CC^{n+1}/\GG_m]}[C_X/\GG_m]\sub \mbox{Vb}( \mathscr{L}_{\cD}^{\oplus (n+1)}).\eeq

Let $\fX{}_X$ be the stack of sections constructed in Proposition \ref{ZA} with $\cZ$ replaced by $\cZ_X$.
There by definition we have a Cartesian diagram
\beq
\begin{CD}
\qme(X,d)@>\sub >>  \qme(\PP^n,d)\\
@VV{\lam_X}V@VV{\lam}V\\
 \fX{}_X @>>>\fX,
\end{CD}
\eeq
Thus $\lam_X$ is also an open embedding.  Then Proposition \ref{deformation} induces
$$\phi_{\qme(X,d)/\cD}:\TT_{\qme(X,d)/\cD}\to \EE_{\qme(X,d)/\cD}:=R^\bullet \pi_{\fX_{X}\ast} \ee\sta \TT_{\cZ_{X}/\cC_{\cD}}.$$
Since  $\EE_{\qme(X,d)/\cD}$ above  is a two-term complex concentrated in $[0,1]$ by Theorem 4.52 in \cite{FKM},
$\phi_{\qme(X,d)/\cD}$ is a relative perfect obstruction theory. It is identical to the perfect obstruction theory of $\qme(X,d)/\cD$ defined by I. Ciocan-Fontanine, B. Kim and D. Maulik
 in \cite[Sect. 4.5]{FKM}.

\section{  $\epsilon$-stable quasimaps with  $P$-fields}
 We enlarge the moduli of $\epsilon$-stable quasimaps to $\PP^n$  by adding $P$-fields, and define its localized virtual cycle. This section is parallel to \cite[Sect.\,  3]{CL}.
 \subsection{Moduli of $\epsilon$-stable quasimaps with $P$-fields}
 Let $M:=\qme(\PP^n,d)$ be the moduli of genus $g$ degree $d$ $\epsilon$-stable quasimaps to $\PP^n$. Let $(\mathcal{C}_M, \pi_M)$  be the universal family of $\qme(\PP^n,d)$, and $\mathscr{L}_{M}$ be the university line bundle. Denote
$$
\sR_M:=\mathscr{L}_{M}^{\vee\otimes m}\otimes\omega_{\mathcal{C}_M/M}.
$$

 Using direct image cone construction, we form $$  \cP=\qme(\PP^n,d)^p:=C(\pi_{M*}\mathscr{P}_{M}) .$$
 It is represented by a DM stack.  We call it the moduli of genus $g$ degree $d$ $\epsilon$-stable quasimaps with $P$-fields. \black

Let $\pi_{\cD}:\cC_{\cD}\to \cD$ be the universal family. Let  $\sL_{\cD}$ be the universal line bundle on $\cC_{\cD}$. Then we have an  invertible sheaf and a vector bundle
$$\sR_{\cD}:=\sL^{\vee\otimes m}_{\cD}\otimes\omega_{\cC_{\cD}/\cD}, \qquad   \fZ\defeq \Vb(\sL^{\oplus (n+1)}_{\cD}\oplus \sR_{\cD})   $$
over $\cC_{\cD}$. If $\fX{}$ be the moduli stack  of sections of $\fZ/\cC_{\cD}$ in Proposition \ref{ZA}, $\cP$ is an open substack of $\fX$. The tautological evaluation morphism $\cC_{\fX{}}\to \fZ$ induces an
evaluation morphism $\ti\ee: \cC_\cpgp\lra \fZ$.
 Let $\pi_{\cpgp}:\cC_{\cP}\to\cP$ and $\mathscr{L}_{\cP}$  be the universal curve and line bundle over $\cP$ defined by the pull-back of $\cC_{\cD}$ and $\sL_{\cD}$. Denote $\sR_{\cP}:=\mathscr{L}_{\cP}^{\vee\otimes m}\otimes\omega_{\mathcal{C}_{\cP}/\cP}$.

\begin{prop}
The pair
$\cpdp\to \cD$ admits a perfect relative obstruction theory
$$\phi_{\cpdp/\cD}:\TT_{\cpdp/\cD}\lra \EE_{\cpdp/\cD}\defeq
R^\bullet\pi_{\cpgp \ast} (\sL^{\oplus (n+1)}_{\cP}\oplus\sR_{\cpgp}).
$$
\end{prop}

\begin{proof} It follows from Proposition \ref{deformation} by applying it to the morphism $\ti\ee$,
using that $\TT_{\fZ/\cC_{\cD}}=\varpi^*\sL^{\oplus (n+1)}_{\cD}\oplus \varpi^*\sR_{\cD}$, where $\varpi:\fZ\to\cC_{\cD}$ denotes the projection.
\end{proof}

\subsection{Construction of the cosection}

  Define a  bundle morphism
\begin{align}\label{co}
\Lambda: \Vb(\sL^{\oplus (n+1)}_{\cD} \oplus\sR_{\cD})\lra \Vb(\omega_{\cC_{\cD}/\cD}),
\quad \Lambda(x,p)=p\cdot q(x),
\end{align}

 for $(x,p)=\bl (x_i),p)\in \Vb(\sL^{\oplus (n+1)}_{\cD} \oplus \sR_{\cD})$.
The product $p\cdot q(x)$  is given by   $\sL^{\otimes m}_{\cD}\otimes\sR_{\cD}\to \omega_{\cC_{\cD}/\cD}$. The $\Lambda$ induces a morphism   
$$
d\Lambda:\TT_{\Vb(\sL^{\oplus (n+1)}_{\cD} \oplus\sR_{\cD})/\cC_{\cD}}\to \Lambda^*\TT_{\Vb(\omega_{\cC_{\cD}/\cD})/\cC_{\cD}}.
$$
 Pull back to $\cC_{\cP}$ via evaluation map $\ti\ee$ and apply $R^{\bullet}\pi_{\cP}$. We have
\beq\label{use}
\begin{CD}
\sigma_1^\bullet:\EE_{{\cpdp}/\cD}\lra R^\bullet\pi_{{\cpdp}\ast}(\ti\ee^\ast
\Lambda^\ast \Omega\dual_{\Vb(\omega_{\cC_{\cD}/\cD})/\cC_{\cD}})\cong R^\bullet\pi_{{\cpdp}\ast}(\omega_{\cC_{\cpdp}/{\cpdp}}).
\end{CD}
\eeq
 It induces the following morphism:
$$\sigma_1:  \Ob_{\cP/\cD}=R^1\pi_{\cpdp \ast}\sL^{\oplus (n+1)}_{\cP} \oplus R^1\pi_{\cpgp\ast} \sR_{\cP}\lra \sO_{\cpdp}.
$$

 A coordinate expression of $\sigma_1$ is as follows. Let
$\mathfrak{u}_{1},\cdots,\mathfrak{u}_{n+1}\in\Gamma (\cC_{\cpgp},\sL_{\cpgp})$ and $\mathfrak{p}\in\Gamma (\cC_{\cpgp},\sR_{\cpgp})$
be the tautological section of $\cpgp$. For any   chart $T\to\cpgp$, let $\cC_T=\cC_{\cpgp}\times_{\cpgp} T$. And let $p$ and $u_{i}$   be the pull-back of $\mathfrak{p}$ and $\mathfrak{u}_{i}$ to $\cC_T$, and  $u:=(u_1,\cdots,u_{n+1})$. Then for arbitrary
$$ \mathring p\in H^1(\cC_T,\sR_{\cpgp})\and
\mathring {u}=(\mathring{u}_{i}) \in H^1(\cC_T, \sL^{\oplus (n+1)}_{ \cpdp}),
$$
\begin{align*}
\sigma_1(\mathring{p},\mathring{u})=p\sum_{i}\frac{\partial q(u)}{\partial u_i}\mathring{u}_{i}+\mathring{p}q(u).
\end{align*}
\subsection{Degeneracy locus of the cosection}  We define the degeneracy loci of $\sigma_1$
$$Deg(\sigma_1):=\Bigl\{\xi\in\cpdp\,\big|\, \sigma_1|_\xi: \Ob_{{\cpgp}/\cD}\otimes_{\sO_{\cpdp}}\CC(\xi)\lra \CC(\xi) \ \text{vanishes}\,\Bigr\}.
$$
 As the definition in section 2, we can embed
$$\qme(X,d)\sub\qme(\PP^n,d)\sub\cpgp,
$$
where the second inclusion is by assigning zero $P$-fields.
\begin{prop}
The degeneracy loci of $\sigma_1$ is $\qme(X,d)\sub \cpdp$; it is proper.
\end{prop}

\begin{proof}
At each $\eta=(C,p_1,\cdots,p_\ell ,L,u,p)\in \cpdp$, where $u=(u_{i})\in H^0(C,L^{\oplus n+1})$, we have $\sigma_1|_{\eta}(\mathring{p},\mathring u)=p\sum_{i}\frac{\partial q(u)}{\partial u_i}\mathring{u}_{i}+\mathring{p}q(u)$. If $\eta\in Deg(\sigma_1)$,
  $\mathring{p}q(u)=0$ for arbitrary $\mathring{p}$ implies $q(u)=0$ by Serre duality.
  Similarly  $p\sum_{i}\frac{\partial q(u)}{\partial u_i}\mathring{u}_{i}=0$ for arbitrary $\{\mathring{u}_i\}$ implies $p\frac{\partial q(u)}{\partial u_i}=0$ for each $i$.  By definition the set of common zeros of  $(u_i)$ is a
   finite set $B\sub C$.  As $C_X-\{0\}$ is smooth, $(\frac{\partial q(u)}{\partial u_1},\cdots,\frac{\partial q(u)}{\partial u_{n+1}})$ is nowhere zero on $C-B$. Thus $p|_{C-B}=0$ and  $p=0$. Therefore  $Deg(\sigma_1)$ is the set of $(C,p_1,\cdots,p_{\ell},L,u,p)$ with $q(u)=0$ and $p=0$. This set is $\qme(X,d)\sub\cpdp$.
\end{proof}

 Let $\fq:\cpdp\to \cD$ be the tautological morphism. We form the  triangle
\beq\label{dt1}
\fq\sta \LL_{\cD}\lra \LL_{\cpdp}\lra \LL_{\cpdp/\cD}\mapright{\zeta} \fq\sta \LL_{\cD}[1].
\eeq
Composed with $\phi_{\cpdp/\cD}:\TT_{\cpdp/\cD}\to \EE_{\cpdp/\cD}$, let
\beq\label{tang}
\eta:=H^0(\phi_{\cpdp/\cD}\circ \zeta\dual):  \fq\sta H^0(\TT_{\cD})\lra
H^1(\TT_{\cpdp/\cD})\lra H^1(\EE_{\cpdp/\cD})=\Ob_{\cpdp/\cD}.
\eeq
  Set  the absolute obstruction sheaf $\Ob_{\cpdp}:=\text{cokernel}\, \eta$. The lemma below follows from exactly same argument as the proof of \cite[Lemm.\, 3.6]{CL}.
  \begin{lemm}\label{cone}
The following composition is trivial:
$$
0=H^1(\sigma_1^\bullet\circ\phi_{\cpdp/\cD}): H^1(\TT_{\cpdp/\cD}) \mapright{}H^1(\EE_{\cpdp/\cD})
\mapright{}
R^1\pi_{\cpdp\ast}\omega_{\cC_{\cpdp}/\cpdp}.
$$
\end{lemm}
\begin{coro}\label{lift}
The cosection $\sigma_1:\Ob_{\cpdp/\cD}\to \sO_\cpdp$ lifts to a  $\bar\sigma_1: \Ob_\cpdp\to \sO_\cpdp$.
\end{coro}
\begin{proof}
  The composition of $\sigma_1$ with \eqref{tang} is the $H^1$ of the composition
$$\fq^*\TT_{\cD}[-1]\mapright{\zeta\dual} \TT_{\cpdp/\cD}\mapright{\phi_{\cpdp/\cD}}\EE_{\cpdp/\cD}\mapright{\sigma_1^\bullet} R^\bullet\pi_{\cpdp\ast}\omega_{\cC_\cpdp/\cpdp},
$$
  which vanishes by Lemma \ref{cone}. 
\end{proof} \black

\subsection{The virtual dimension and  the virtual cycle}
  Pick arbitrary closed point $\xi=(C,\{p_j\}_{j=1}^\ell , L,\phi, p) $ of $\cpgp$. The virtual dimension   of $\cpgp/\cD$ is
 $$\rank (\EE_{\cpgp/\cD}\otimes_{\sO_{\cpgp}}\CC(\xi))=
\chi(C,L^{\oplus (n+1)})+\chi(C,L^{\vee \otimes m}\otimes \omega_C)=(n+1-m)d+n(1-g).
$$
 One then calculates
$$\dim {\cD} =\dim \cD/\barM_{g,\ell}  + \dim \barM_{g,\ell}=g-1+3g-3+\ell=4(g-1)+\ell,
$$
 $$\delta:=\text{vdim}\, \cpgp=\dim \cD +\text{vdim}\,\cpgp/\cD=(n+1-m)d+(n-4)(1-g)+\ell.$$ \black
 Let $\cU\sub \cpgp$ be the locus where $\sigma_1$ is surjective; and denote  $Deg(\sigma_1)=\cpgp-\cU
$.  Since $\sigma_1|_\cU$ is surjective,
it induces a surjective bundle homomorphism
\begin{align*} \sigma_1|_\cU: h^1/h^0(\EE_{\cpgp/\cD})\times_\cP\cU\lra \CC_\cU,
\end{align*}
where $\CC_\cU=\CC\times\cU$. Let $\ker(\sigma_1|_\cU)$ be the kernel bundle-stack of $\sigma_1|_\cU$.
 Endow
\beq\label{cone-stack}
h^1/h^0(\EE_{\cpgp/\cD})(\sigma_1)=\bl h^1/h^0(\EE_{\cpgp/\cD})\times_{\cpgp} Deg(\sigma_1)\br \cup
\ker(\sigma_1|_\cU).
\eeq
 with the reduced structure. It is a closed substack of $h^1/h^0(\EE_{\cpgp/\cD})$.

 By \cite[Thm.\, 5.1]{KL} the normal cone cycle $[\bC_{\cpgp/\cD}]\in Z\lsta h^1/h^0(\EE_{\cpgp/\cD})$
lies in $Z\lsta h^1/h^0(\EE_{\cpgp/\cD})(\sigma_1)$.  Kiem-Li's localized Gysin map  gives
$$0^!_{\bar\sigma_1,\mathrm{loc}}: A\lsta h^1/h^0(\EE_{\cpgp/\cD})(\sigma_1)\lra A_{\ast} Deg(\sigma_1).
$$
\begin{defi-prop}
We define the localized virtual cycle of $\cpgp$ to be
$$[\cpgp]\virt=[\qme(\PP^n,d)^p]\virt:=0^!_{\bar\sigma_1,\mathrm{loc}}([\bC_{\cpgp/\cD}])\in A_{}
\qme(X,d).
$$
\end{defi-prop}

\section{Degeneration of moduli of $\epsilon$-stable quasimaps with $P$-fields}

In the following sections, we  use  degenerations   in \cite{CL} to prove  $[\cpgp]\virt$ coincides up to a sign with $[\qme(X,d)]^{\text{vir}}$. 
  The setup is close to   \cite{CL} but
 the perfectness of family obstruction theories requires the quasimap conditions(see Proposition \ref{ob-WV}).
 We also give a new proof of the  constancy of virtual cycles in the degeneration (Theorem \ref{constancy}), which greatly simplifies that in \cite{CL1}.  

 Let $(\mathcal{C}_M, \pi_M)$ be the universal family of $\qme(\PP^n,d)$, and let $\mathscr{L}_{M}$ be the university line bundle as before. We form a separated DM stack 
$$
F=C(\pi_{M*}(\sL_{M}^{\otimes m}\oplus \sO_{\cC_{M}})).
$$

 We define $\cK_{g,\ell}^{\ep}(V,d)$ to be the subgroupoid of $F$, such
 that $\cK_{g,\ell}^{\ep}(V,d)(S)$ is the set of all families:
 \beq\label{obj}(\pi: \cC_S \rightarrow S, \{ p_i: S \rightarrow \cC_S\}_{i=1,\dots,\ell}, \sL_S, \{u_i\}_{i=1}^{n+1},t,y)\in F(S)\eeq
 subject to the equation
 \begin{align*}q(u_1,\cdots,u_{n+1})-ty=0\in \Gamma(\cC_S,\sL_S^{\otimes m}),\end{align*}
  where $t\in \Gamma(\cC_S,\sO_{\cC_S})$ and $y\in\Gamma(\cC_S,\sL_S^{\otimes m})$.
 Clearly $\cK_{g,\ell}^{\ep}(V,d)$ is a closed substack of $F$, and thus is a separated DM stack.

 Now we introduce the stable quasimaps with $p$-field. Let $(\widetilde{\cC},\ti \pi)$ be the universal family of $\cK_{g,\ell}^{\ep}(V,d)$, and let $\ti \sL$ be the universal line bundle over $\widetilde{\cC}$. Let
$\ti \sR=\ti \sL^{\vee\otimes m}\otimes \omega_{\widetilde{\cC}/\cK_{g,\ell}^{\ep}(V,d)}$ be the tautological   invertible sheaves. Define the moduli of stable morphisms coupled
with $P$-fields to be
$$\cvgp\defeq \cK_{g,\ell}^{\ep}(V,d)^p\defeq C(\ti \pi_{\ast}\ti \sR_{}),
$$
the direct image cone. Sending \eqref{obj} to $S\mapright{t} \Ao$ induces a map $\cvgp\to\Ao$, with fibers canonically described as
$$\cvgp_c:=\cvgp\times_{\Ao} c\cong \cpg,\quad \ \ \ c\ne 0; $$
$$\cvgp_0:=\cvgp\times_{\Ao} 0=:\cN \quad \ \ \ c=0.$$

 By construction $\cvgp_0$ represents the groupoid  where
 $\cvgp_0(S)$ is the set of all $(\cC_S, \sL_S, \cdots,p,y)$
 where $(\cC_S, \sL_S, \cdots)\in \qme(X,d)(S)$,  $p\in\Gamma(\cC_S,\sL_S^{\vee\otimes m}\otimes\omega_{\cC_S/S})$ \black and $y\in\Gamma(\cC_S,\sL_S^{\otimes m}).$

\subsection{The evaluation maps}
 We construct a natural  obstruction theory
of $\cV$ relative to $\DD:=\cD\times\Ao$. Denote  the universal
curve by
$$\pi_{\DD}:\cC_{\DD}\defeq \cC_{\cD}\times\Ao\lra {\cD}\times\Ao=\DD,
$$ and $\sL_{\DD}$  the pull-back of $\sL_{\cD}$ via $\cC_{\DD}\to\cC_{\cD}$. We have a bundle over $\cC_{\DD}$
\begin{align*}\Vb(\sL_{\DD}^{\oplus (n+1)})\times_{\cC_{\DD}} \Vb(\sL^{\otimes m}_{\DD})\lra \cC_{\DD}.
\end{align*}

Let $E_{2}=\CC\lAo$ (resp. $E_1=\CC\lAo^{\oplus (n+1)}$)
be the trivial line bundle (resp. rank $n+1$ trivial vector bundle) over $\Ao$.
We consider the rank $n+2$ bundle
$$\pr_{\Ao}: E_1\times\lAo E_2 \mapright{}\Ao
$$
with the $\Gm$-action: $\Gm$ acts on the base $\Ao$ trivially and acts on fibers of $E_1$ via $g\cdot(x_1,\cdots,x_{n+1})=(gx_1,\cdots,gx_{n+1})$, and acts on fibers of $E_2$ as  $ g\cdot( y_0)=(g^m y_0)$.\black

 Denote  $\Delta:=q(x_i)- t\cdot y_0
$ and let $C(V)=(\Delta=0)\sub  E_1\times_{\Ao} E_2$, where $t$ denotes the coordinate of $\Ao$.
There is a canonical morphism
$$
\Vb(\sL_{\DD}^{\oplus (n+1)})\times_{\cC_{\DD}} \Vb(\sL^{\otimes m}_{\DD})\lra
 [(E_1\times\lAo E_2)/\GG_m].
$$
We define
\begin{align*}
\cX'= \bigg( \Vb(\sL^{\oplus (n+1)}_{\DD})\times_{\cC_{\DD}} \Vb(\sL^{\otimes m}_{\DD})\bigg) \times_{[(E_1\times\lAo E_2)/\GG_m]} [C(V)/\GG_m],
\end{align*}
and
\beq\label{ZZ}
\cX=\cX'\times_{\cC_{\DD}} \Vb(\sR_{\DD}).
\eeq
Let $\pi_{\cV}:\cC_{\cV}\to \cV$ be the universal family.
The natural evaluation morphism
$$e:\mathcal{C}_{F}\lra \Vb(\sL^{\oplus (n+1)}_{\DD})\times_{\cC_{\DD}} \Vb(\sL^{\otimes m}_{\DD}),$$
(where $\cC_F$ is the universal family of $F$) restricts to give the evaluation morphism
$$
\ee_{\cV}: \cC_{\cvgp}\lra \cX.
$$

\subsection{The obstruction theory of $\cV/\DD$}

We begin with a description of the tangent complex $\TT_{\cX'/\cC_{\DD}}$. Let $\varrho: \cX'\to \cC_{\DD}$ be the
tautological projection.
By the defining equation of $C(V)$,
$$
\TT_{C(V)/\AA} \cong [\sO^{\oplus (n+1)}_{C(V)}\oplus \sO_{C(V)}\mapright{d\Delta} \sO_{C(V)}],
$$
where $d\Delta$ at $((x_i),y_0, t)\in C(V)$ sends $((\mathring x_i), \mathring y_0)$ to $\sum \frac{\partial q(x_i)}{\partial x_i} \mathring x_i- t\mathring y_0$.
Thus
$$\TT_{\cX'/\cC_{\DD}}\cong\TT_{\fP\times_{\GG_m} C(V)/\cC_{\DD}}\cong \bigg[\varrho\sta \sL^{\oplus (n+1)}_{\DD}\oplus \varrho\sta \sL^{\otimes m}_{\DD}\mapright{d\overline{\Delta}} \varrho\sta\sL^{\otimes m}_{\DD}\bigg] ,$$ \black
where $d\overline{\Delta}$ at $((z_i),y, t)\in \cX'$
sends $((\mathring z_i), \mathring y)$ to $\sum \frac{\partial
q(z_i)}{\partial z_i} \mathring z_i- t\mathring y$. Let
$\sL_{\cV}$ be the universal line bundle over $\cC_{\cV}$, then
\begin{align*} \ee^*_{\cV}\TT_{\cX/\cC_{\DD}}\cong
\bigg[\sL^{\oplus (n+1)}_{\cV}\oplus  \sL^{\otimes
m}_{\cV}\mapright{d\kappa} \sL^{\otimes m}_{\cV}\bigg]\oplus
\big[\sL^{\vee\otimes m}_{\cV}\otimes\omega_{\cC_{\cV}/\cV}\lra
0\big], \end{align*}
 where $d\kappa$ restricted to $((\phi_i),\ti
t,b,p)\in \cC_{\cV}$ sends $((\mathring \phi_i), \mathring b)$ to
$\sum \frac{\partial q(\phi_i)}{\partial \phi_i} \mathring
\phi_i- t\mathring b$. Denote
$$
\cK^{\bullet}_1=\bigg[\sL^{\oplus (n+1)}_{\cV}\oplus  \sL^{\otimes m}_{\cV}\mapright{d\kappa} \sL^{\otimes m}_{\cV}\bigg].
$$
\begin{lemm}\label{tor}
Let $C$ be a geometric fiber of $\cC_{\cV}\to \cV$ over arbitrary closed point $\xi\in\cV$, then $\cH^1(\ee^*_{\cV}\TT_{\cX/\cC_{\DD}}|_{C})$ is a torsion sheaf on $C$.
\end{lemm}
\begin{proof}
By the definition of $\cV$, for each point $\xi=((u_i), t,y,p)\in\cV$
there exists a finite set $B\subset C$ such that $(u_i(z))\neq 0$ for every $z\notin B$. Therefore $d\kappa$ is surjective when restricts on $C\setminus B$.
\end{proof}

\begin{prop}\label{ob-WV}
The   $\cvgp\to\DD$ has a   relative perfect obstruction theory
$$\phi_{\cvgp/\DD}: \TT_{\cvgp/\DD}\lra \EE_{\cvgp/\DD}:=
R^\bullet \pi_{\cV\ast} \ee_{\cV}^\ast \TT_{\cX/\cC_{\DD}}.  $$
Its specialization at $c\ne 0\in \Ao$ (resp. $0\in \Ao$) is the
$\phi_{\cpgp/\cD}$ (resp. $\phi_{\cN/\cD}$).
\end{prop}
\begin{proof} Since $\cvgp$ is an open substack of the direct image cone
$C(\pi_{\DD\ast}\cX)$, Proposition \ref{deformation} provides a canonical  obstruction theory $\phi_{\cvgp/\DD}$. Let $C$ be a geometric fiber of $\cC_{\cV}\to \cV$ over arbitrary closed point $\xi\in\cV$. As $\ee^*_{\cV}\TT_{\cX/\cC_{\DD}}|_{C}$ is  two-term, the  sequence $$\cH^0(\ee^*_{\cV}\TT_{\cX/\cC_{\DD}}|_{C})\to\ee^*_{\cV}\TT_{\cX/\cC_{\DD}}|_{C}\to\cH^1(\ee^*_{\cV}\TT_{\cX/\cC_{\DD}}|_{C})[-1]\mapright{+}$$
is a distinguished triangle.

 Since $H^1(C,\cH^1(\ee^*_{\cV}\TT_{\cX/\cC_{\DD}}|_{C}))$ vanishes by Lemma \ref{tor}, $H^2(\ee^*_{\cV}\TT_{\cX/\cC_{\DD}}|_C)=0$. Thus $
R^\bullet \pi_{\cV\ast} \ee_{\cV}^\ast \TT_{\cX/\cC_{\DD}}$ is perfect of amplitude $[0,1]$.

 When $c\ne0$ using $\iota_c: \cvgp=\cvgp\times_{\Ao}c \to \cvgp$, the functoriality of the construction implies that
$\phi_{\cpgp/\cD}$  is the composition of
$\TT_{\cpgp/\cD}\to \iota_c\sta \TT_{\cvgp/\DD}$ with
$$\iota_c\sta(\phi_{\cvgp/\ti{\cD}_{g,k}}):\iota_c\sta \TT_{\cvgp/\DD}\lra \iota_c\sta \EE_{\cvgp/\DD}\cong\EE_{\cpgp/\cD}.
$$
In case $c=0$, we define $\EE_{\cN/\cD}\defeq \iota_0\sta \EE_{\cvgp/\DD}$. This proves the Proposition.
\end{proof}

 Let $\cC_{\cN}=\cC_{\cV}\times_{\cV}\cN$, and $\pi_{\cN}:\cC_{\cN}\to\cN$ be the restriction of $\pi_{\cV}$. If $c= 0$,
$$\cK_1^{\bullet}\cong \big[\sL^{\oplus (n+1)}_{\cV}\mapright{d\kappa_1} \sL^{\otimes m}_{\cV}\big]|_{\cC_{\cN}}\oplus  \big[\sL^{\otimes m}_{\cV}\lra 0\big]|_{\cC_{\cN}},$$
where $d\kappa_1$ restricted to $\cC_{\cN}$ sends $(\mathring \phi_i)$ to $\sum \frac{\partial q(\phi_i)}{\partial \phi_i} \mathring \phi_i$.
Denote by
$$
\cK_2^{\bullet}=\big[\sL^{\oplus (n+1)}_{\cV}\mapright{d\kappa_1} \sL^{\otimes m}_{\cV}\big]|_{\cC_{\cN}}.
$$

Then
\begin{align*}H^1(\EE_{\cN/\cD})\cong R^1\pi_{\cN*}(\cK_2^{\bullet})\oplus R^1\pi_{\cN*}(\sL^{\otimes m}_{\cV}|_{\cC_{\cN}})\oplus R^1\pi_{\cN*}(\sL^{\vee\otimes m}_{\cV}\otimes\omega_{\cC_{\cV}/\cV}|_{\cC_{\cN}}).\end{align*}

\subsection{Family cosection of $\Ob_{\cvgp/\DD}$}  Let $h$ be a bi-linear morphism of bundles
$$h: \Vb(\sL^{\oplus (n+1)}_{\DD}\oplus \sL^{\otimes m}_{\DD}\oplus \sR_{\DD})\mapright{(\pr_2,\pr_3)}
\Vb(\sL^{\otimes m}_{\DD})\times_{\cC_{\DD}}\Vb(\sR_{\DD})\lra \Vb(\omega_{\cC_{\DD}/\DD}).
$$
 where $\pr_i$ is the $i$-th projection, and the second arrow is induced by
$\sL^{\otimes m}_{\DD}\otimes\sR_{\DD}\to\omega_{\cC_{\DD}/\DD}$. Using that the family $\cX\to \cC_{\DD}$ in \eqref{ZZ}
is a subfamily
$$\cX
  \sub \Vb(\sL^{\oplus (n+1)}_{\DD}\oplus \sL^{\otimes m}_{\DD}\oplus \sR_{\DD}),
$$
composing with $h$, we obtain a $\cC_{\DD}$-morphism $\cX\lra \Vb(\omega_{\cC_{\DD}/\DD}).$ \black
\begin{lemm}
The homomorphism $\cX\lra \Vb(\omega_{\cC_{\DD}/\DD})$ induces a homomorphism
$$\sigma^\bullet: \EE_{\cvgp/\DD}\lra  R^{\bullet}\pi_{\cvgp\ast}\omega_{\cC_{\cvgp}/\cvgp}
$$
whose restriction to $\cvgp\times_{\Ao} c\cong \cpgp$, $c\ne 0$, is  proportional (by an element in $\CC\sta$)
to $\sigma_1^\bullet$ in \eqref{use}.
\end{lemm}

\begin{proof}
The proof is exactly as  in Section 3.2. We omit it here.
\end{proof}

We denote
\beq\label{si-V}
\sigma=H^1(\sigma^\bullet): \Ob_{\cvgp/\DD}\defeq H^1(\EE_{\cvgp/\DD})\lra R^1\pi_{\cvgp\ast}\omega_{\cC_{\cvgp}/\cvgp}
\cong \sO_{\cvgp}.
\eeq
 For $c\ne 0$, the restriction of
$\sigma$  over $\cvgp\times_{\Ao} c\cong \cpgp$ with $c\ne 0$.
 equals to $\sigma_1$. For $c=0$ the restriction $\sigma_0:=\sigma|_{c=0}$ admits the following simple expression.

   Every closed point $\xi\in\cN$  can be represented by
$$((\phi_i),b,p,0)\in  H^0(L^{\oplus (n+1)})\times H^0(L^{\otimes m})\times H^0(L^{\vee\otimes m}\otimes \omega_C)\times H^0(C,\sO_C)
$$
where $(C,p_1,\cdots,p_\ell ,L)\in \cD$ denotes the point under $\xi$. The restriction  $$\sigma_0|_\xi: \Ob_{\cN/\cD}|_\xi \to \CC$$
is  the composite of
 $$ \Ob_{\cN/\cD}|_\xi\mapright{\varsigma|_\xi} H^1(L^{\oplus (n+1)})\oplus H^1(L^{\otimes m})\oplus H^1(L^{\vee\otimes m}\otimes\omega_C)
$$
with the pairing
$$H^1(L^{\oplus (n+1)})\oplus H^1(L^{\otimes m})\oplus H^1(L^{\vee\otimes m}\otimes\omega_C)\lra H^1(\omega_C)
$$
defined via $((\mathring \phi_i), \mathring b, \mathring p)\mapsto  \mathring b\cdot p+b\cdot \mathring p$.

Thus $\sigma_0|_\xi=0$ if and only if $p=0$ and $b=0$. The loci $Deg(\sigma_0)$ is $\cN\cap (p=b=0)=\qme(X,d)$. Since $\sigma|_{c\neq 0}\cong \sigma_1$ we conclude
$$Deg(\sigma)=\qme(X,d)\times\Ao\sub \cvgp.$$

 Let $\ti \fq:\cvgp\to \DD$ be the projection.  The
 natural
$\ti\fq^\ast\TT_{\DD}\to \TT_{\cvgp/\DD}[1]$
composed with $\phi_{\cvgp/\DD}:\TT_{\cvgp/\DD}\to \EE_{\cvgp/\DD}$ gives
$\eta: \ti\fq^\ast\TT_{\DD}\lra  \EE_{\cvgp/\DD}[1].$ \black Set
\begin{align*}
\Ob_{\cvgp}\defeq \coker\{ H^0(\eta): \ti\fq^\ast\Omega_{\DD}\dual\lra  H^1(\EE_{\cvgp/\DD})\}.
\end{align*}

 By same reasons as Lemma \ref{cone} and Corollary \ref{lift}, we have the following.
\begin{lemm}\label{family-cone}
The composite vanishes
\begin{align*}
\ti\fq\sta H^0(\TT_{\DD}) \mapright{H^0(\eta)} H^1(\EE_{\cvgp/\DD})\mapright{\sigma}
R^1\pi_{\cvgp\ast}\omega_{\cC_{\cvgp}/\cvgp}.
\end{align*}
Thus the cosection $  \sigma:\Ob_{\cvgp/\DD}\to \sO_{\cvgp}$ lifts to a cosection $\bar\sigma: \Ob_{\cvgp}\to \sO_{\cvgp}$.
\end{lemm}

\subsection{The constancy of the virtual cycles}

 One verifies directly the virtual dimension of $\cvgp$ is $\delta+1$.
Using Lemma \ref{family-cone},   following the convention introduced \black
in Subsection 3.3, we denote by
$$h^1/h^0(\EE_{\cV/\DD})(\sigma)\sub h^1/h^0(\EE_{\cV/\DD})
$$
the kernel of a cone-stack morphism $h^1/h^0(\EE_{\cV/\DD})\to\CC_\cV$
induced by $\bar\sigma$ defined as in
\eqref{cone-stack}.\footnote{ It is $h^1/h^0(\EE_{\cV/\DD})$ along the degeneracy loci and is the kernel of
$h^1/h^0(\EE_{\cV/\DD})\to\sO_{\cvgp}$ induced by $\sigma$ away from the degeneracy loci.} By \cite{KL} we have the localized Gysin map
$$0^!_{\bar\sigma,\mathrm{loc}}: A\lsta h^1/h^0(\EE_{\cV/\DD})(\sigma)\lra A\lsta (\qme(X,d)\times\Ao).
$$

\begin{defi} We define the localized virtual cycle of $(\cvgp,\bar\sigma)$ be
$$[\cvgp]\virt:=0^!_{\bar\sigma,\mathrm{loc}}([\bC_{\cvgp/\DD}])\in A_{\delta+1} (\qme(X,d)\times\Ao).
$$
 \end{defi}
Let $\jmath_c:c\to \AA^1$ be the inclusion, then we have the following theorem.
\begin{theo}\label{constancy} $\jmath_c^![\cvgp]\virt=[\cvgp_c]\virt$.
 \end{theo}
 \begin{proof} Let $\cvgp/\cD$ be the composition of $\cvgp\to\DD$ with the projection $\DD=\cD\times\Ao\to \cD$.
 Let $\ff$ be the composition of  $ \ti \fq\sta\TT_{\DD/\cD}[-1]\to \TT_{\cvgp/\DD}$ with $\phi_{\cvgp/\DD}$. We have a diagram of distinquished triangles
  \beq\label{dia00}
  \begin{CD}
  \ti \fq\sta\TT_{\DD/\cD}[-1]@>{\ff}>>\EE_{\cvgp/\DD}@>{\fg}>>\EE_{\cvgp/\cD}\\
  @AA{||}A@AA{\phi_{\cvgp/\DD}}A@AA{\phi_{\cvgp/\cD}}A\\
  \ti \fq\sta\TT_{\DD/\cD}[-1]@>>>\TT_{\cvgp/\DD}@>>>\TT_{\cvgp/\cD}
  \end{CD}
  \eeq
   where $\EE_{\cvgp/\cD}$ is  the mapping cone of  $\ff$, and  $\phi_{\cvgp/\cD}$  is given by mapping cone axiom. From
   diagram \ref{dia00} we have the following commutative diagram
 $$
  \begin{CD}
  0@>>> h^{0}(\EE_{\cvgp/\DD})@>>> h^{0}(\EE_{\cvgp/\cD})@>>> \\
  @.@AA{\phi^0_{\cvgp/\DD}}A@AA\phi^0_{\cvgp/\cD}A\\
  0@>>>h^{0}(\TT_{\cvgp/\DD})@>>>h^{0}(\TT_{\cvgp/\cD})@>>>
  \end{CD}
  $$
  $$
  \begin{CD}
  h^{1}(\ti \fq\sta\TT_{\DD/\cD}[-1])@>>>h^{1}(\EE_{\cvgp/\DD})@>>> h^{1}(\EE_{\cvgp/\cD})@>>>0 \\
  @AA{||}A@AA{\phi^1_{\cvgp/\DD}}A@AA{\phi^1_{\cvgp/\cD}}A\\
 h^{1}(\ti \fq\sta\TT_{\DD/\cD}[-1])@>>> h^{1}(\TT_{\cvgp/\DD})@>>>h^{1}(\TT_{\cvgp/\cD})@>>>0.
  \end{CD}
 $$
 By chasing diagrams \black we know $\phi_{\cvgp/\cD}$ is a relative perfect obstruction theory. As in the proof of \cite[Prop.\, 3]{KKP}, we have a short exact sequence of cone stack
$$
h^{1}/h^0(\ti \fq\sta\TT_{\DD/\cD}[-1])\to\fC_{\cvgp/\DD}\to\fC_{\cvgp/ \cD},
$$
as well as the similar exact sequence relating $h^{1}/h^0(\EE_{\cvgp/\DD})$ with $ h^{1}/h^0(\EE_{\cvgp/\cD})$. Note $\fC_{\cvgp/\DD}$ is the pullback of $\fC_{\cvgp/ \cD}$ along the projection $h^{1}/h^0(\EE_{\cvgp/\DD})\to h^{1}/h^0(\EE_{\cvgp/\cD})$. Let $\fq_1:\DD\to \cD$ and $\fq_2:\DD\to \AA^1$ be the projections, then $\LL_{\DD}=\fq_1^*\LL_{\cD}\oplus\fq_2^*\Omega_{\AA^1}$. Therefore $\ti \fq\sta\TT_{\DD/\cD}[-1]=\ti\fq\sta\fq_2^*\Omega^{\vee}_{\AA^1}$, then by Lemma \ref{family-cone}, $\sigma$ descended to a cosection $\sigma_{\cvgp/\cD}$ for $\Ob_{\cvgp/\cD}:=H^1(\EE_{\cvgp/\cD})$. Therefore the cosection localized virtual cycle of such $\cvgp/\cD$ coincide with that defined by $0^!_{\sigma,\text{loc}}[\fC_{\cvgp/\DD}]$.

   Restrict \eqref{dia00} on $\cvgp_c$  we obtain a diagram
  \beq\label{dia5}
  \begin{CD}
   (\ti \fq\sta\TT_{\DD/\cD}[-1])|_{\cvgp_c}@>>>\EE_{\cvgp/\DD}|_{\cvgp_c}=\EE_{\cvgp_c/\cD}@>{\fg}>>\EE_{\cvgp/\cD}|_{\cvgp_c}\\
  @AAA@AAA@AA{\phi_{\cvgp/\cD}|_{\cvgp_c}}A\\
  (\ti \fq\sta\TT_{\DD/\cD}[-1])|_{\cvgp_c}@>>>\TT_{\cvgp/\DD}|_{\cvgp_c}@>>>\TT_{\cvgp/\cD}|_{\cvgp_c}\\
  @AA{\alpha}A@AA{\beta}A@AA{||}A\\
  \TT_{\cvgp_c/\cvgp}@>>>\TT_{\cvgp_c/\cD}@>>>\TT_{\cvgp/\cD}|_{\cvgp_c}
  \end{CD}
  \eeq
 where  all arrows in the second and the third rows are the natural morphisms induced by the fiber product of $\cD$-stacks
 $\cvgp_c=\cD\times_{\DD}\cvgp$ (here $\cD\to \DD$ is the product of $c\mapright{\jmath_c}\Ao$ with $\DD\to \Ao$).
 The morphism  $\alpha$ is constructed by mapping cone axiom applied to the second and the third rows of \eqref{dia5}.
  The first and the third row of \eqref{dia5} then give  the compatibility  between the deformation theories of $\cvgp/\cD$ and $\cvgp_c/\cD$   required in \cite[(5.1)]{KL}. Apply  \cite[Thm.\, 5.2]{KL}, we have $\jmath_c^![\cvgp]\virt=[\cvgp_c]\virt.$
 \end{proof}
\begin{coro}\label{shriek}
Under the shriek operation of cycles, for $c\ne 0$,
$$\jmath_c^{!} ([\cvgp]\virt)=[\cpgp]\virt\in  A_{\delta} \qme(X,d),
\quad
\jmath_0^{!}([\cvgp]\virt)=[\cN]\virt\in  A_{\delta} \qme(X,d).
$$
\end{coro}

Here $[\cN]\virt$ is the localized virtual cycle using the obstruction theory of
$\cN$ induced by the restricition of $\phi_{\cvgp/\DD}$ (Prop \ref{ob-WV}) and the cosection
$\sigma_0=\sigma|_{\cN}$.

\section{The virtual cycles of   $\cN$ and $\qme(X,d)$}
\def\lcng{_{\cng}}
  In the special case of \cite[Thm.\, 5.7]{CL} there are no markings and  $X$ is the quintic CY hypersurface of $\Pf$. These imply that the  virtual dimension $\delta=0$, and \cite[Thm.\, 5.7]{CL} proved the numerical identity $\deg [\cN]\virt=\deg [\barM_{g}(X,d)]\virt,$ where the proof uses the vanishing of virtual dimension essentially.

   In this section we drop all the conditions and prove a more general property of virtual cycles(Proposition \ref{main1}) to compare $[\cN]\virt$ with $[\qme(X,d)]\virt$.

  \subsection{The virtual cycle of $\cN$}
\def\cqg{\cQ}
By the construction of section 4,
we have an evaluation morphism
$$\ee_{\cN}: \cC_\cN\lra \cX_0=\cX\times\lAo 0.
$$
 and a perfect relative obstruction theory
\begin{align*}
\phi_{\mathcal{N}/\cD}: \TT_{\mathcal{N}/\cD}\lra \EE_{\cngp/\cD}=\iota_0\sta \EE_{\cvgp/\DD}\cong R^{\bullet}\pi_{\mathcal{N}\ast} \ee_\cN^\ast \TT_{\cX_0/\cC_{\cD}}
\end{align*}
 which is
the restriction of $\phi_{\cV/\cD\times \mathbb{A}^1}$ to the fiber over $0\in \Ao$.

\begin{prop}\label{degQ}
The cosection $\sigma_0$ (c.f.\eqref{si-V}) lifts to a cosection
$\bar\sigma_0:\Ob_\cN\to \sO_\cN$. The degeneracy 
loci of   $\bar\sigma_0$ is $\qme(X,d)\sub \cN$,  and is
proper.
\end{prop}

\begin{proof}
This follows  from the  description of $Deg(\sigma)$ and Lemma \ref{family-cone}.
\end{proof}

 As $[\cV]\virt$ has dimension $\delta+1$, by Corollary \ref{shriek},
   the  dimension of $[\cN]\virt$ is $\delta$.

\begin{defi-prop} We define the localized virtual cycle of $\cN$ to be
$$[\cN]\virt:=0^!_{\bar\sigma_0,\mathrm{loc}}([\bC_{\cN/\cD}])\in A_{\delta} \qme(X,d);
$$
\end{defi-prop}

\subsection{Virtual cycles of $\cN$ and $\qme(X,d)$ }
  In the fiber diagram
\beq\label{v-v}
\begin{CD}
\cngp@>{\gamma}>>  \cB\defeq C(\pi_{\cD\ast}(\sL^{\otimes m}_{\cD}\oplus\sR_{\cD}))\\
@VV{v}V@VVV\\
 \cqg  \black \defeq \qme(X,d) @>>>\cD,
\end{CD}
\eeq
the morphism $\gamma$ pulls back the relative perfect obstruction theory
\begin{equation}\label{relll}
\TT_{{\cB}/\cD}\lra  \EE_{{\cB}/\cD}
\end{equation}
to
$$\phi_{\cngp/\cqg}: \TT_{\cngp/\cqg}\lra\EE_{\cngp/\cqg}:=\gamma^\ast \EE_{{\cB}/\cD}.
$$

 The evaluation maps of $\cngp$ and $\cqg$ fit in (for $\cZ_X$ c.f.  \eqref{cZX})
$$
\begin{CD}
\cC_\cngp@>e_{\mathcal{N}}>>\cX_0@>>>\Vb(\sL^{\otimes m}_{\cD})\times_{\cC_{\cD}}\Vb(\sR_{\cD})\\
@VV{\upsilon_\cC}V @VVV@VVV\\
\cC_\cqg@>e_{\mathcal{Q}}>> \mathcal{Z}_X@>>>\cC_{\cD}
\end{CD}
$$
where the right square is a fiber product of flat morphisms, and $\upsilon_\cC$ is induced by
the vertical arrow $v$ in diagram \eqref{v-v}.

 From the above diagram we have a morphism between exact triangles
$$
\begin{CD}
e_\cngp^\ast \TT_{\cX_0/\mathcal{Z}_X}@>>>e_\cngp^\ast \TT_{\cX_0/\cC_{\cD}}@>>>\upsilon_\cC^\ast e_\cqg^\ast \TT_{\mathcal{Z}_X/\cC_{\cD}}@>{+1}>>\\
@AAA@AAA@AAA\\
\TT_{\cC_\cngp/\cC_\cqg}@>>>\TT_{\cC_\cngp/\cC_{\cD}}@>>> \upsilon_\cC^\ast \TT_{\cC_\cqg/\cC_{\cD}}@>{+1}>>\\
\end{CD}
$$
\black

by the projection formula we have  $\EE_{\cngp/\cqg}\cong R^{\bullet}\pi_{\cngp\ast}e_\cngp^\ast \TT_{\cX_0/\mathcal{Z}_X}$, and
$$
\begin{CD}
\EE_{\cngp/\cqg}@>>>\EE_{\cngp/\cD}@>{h}>>\upsilon^\ast \EE_{\cqg/\cD}@>{+1}>>\\
@AA{\phi_{\cngp/\cqg}}A@AA{\phi_{\cngp/\cD}}A@AA{\upsilon^\ast\phi_{\cqg/\cD}}A\\
\TT_{\cngp/\cqg}@>>>\TT_{\cngp/\cD}@>>>\upsilon^\ast \TT_{\cqg/\cD} @>{+1}>>\\
\end{CD}.
$$
 Thus $\phi_{\cngp/\cqg}: \TT_{\cngp/\cqg}\lra\EE_{\cngp/\cqg}$
is a perfect obstruction theory. Composing the cosection
$\sigma_0:\Ob_{\cngp/\cD}\to\sO_{\cngp}$ with
$H^1(\EE_{\cngp/\cqg})\to H^1(\EE_{\cngp/\cD})$, we obtain
$$\tilde{\sigma}_0:\Ob_{\cngp/\cqg}:=H^1(\EE_{\cngp/\cqg})\lra \sO_{\cngp}.
$$
Using the argument parallel to  Proposition \ref{degQ}, one sees that the degeneracy loci of $\tilde{\sigma}_0$ equals $\cqg\sub\cngp$.

We are in a situation that fits the setup of \cite[Sect.\, 2.2]{CKL}, and we apply \cite[Def.\, 2.8]{CKL} below.
Let $\bC_{\cngp/\cqg}$ be the intrinsic normal cone of $\cngp$ relative to $\cqg$.
Then
$\bC_{\cngp/\cqg}$ is a closed substack of $h^1/h^0(\EE_{\cngp/\cqg})$ by \eqref{relll}.
We denote
$$\Omega:=h^1/h^0(\EE_{\cngp/\cqg})\times_{\cN}\cQ\cup\ker\{\tilde{\sigma}_0:  h^1/h^0(\EE_{\cngp/\cqg})\lra \CC_{\cngp}\}.
$$

The virtual pullback morphism of cosection localized classes
$$\upsilon^!_{\text{loc}}: A\lsta\cqg\lra A\lsta\cqg
$$\black
defined as the composite of
\begin{equation}\label{com}
A\lsta\cqg\lra A\lsta\bC_{\cngp/\cqg}\lra A\lsta \Omega
\mapright{0^!_{h^1/h^0(\EE_{\cngp/\cqg}),\mathrm{loc}}} A\lsta\cQ,
\end{equation}
where the first arrow sends $\sum n_i[V_i]$ to $\sum n_i [\bC_{V_i\times_\cqg \cngp/V_i}]$,
the second arrow is induced by the inclusion $\bC_{\cngp/\cqg}\sub \Omega$, and the third   is the localized Gysin map defined in \cite{KL} via $\tilde{\sigma}_0$. By \cite[Thm.\, 2.9]{CKL} we have
\begin{lemm}\lab{pullback}
$$\upsilon^!_{\text{loc}}([\cqg]^{vir})=[\cN]\virt\in A_{\delta} \cQ.
$$
\end{lemm}

\subsection{A general comparison}
 To simplify $\upsilon^!_{\text{loc}}([\cqg]^{vir})$, we establish a general comparison, from which the 
Theorem  \ref{theo1} easily follows.

 Let $D$ be a smooth Artin stack over $\CC$, $\pi:\cC\to D$ be a flat family of connected, nodal, genus $g$ curves\footnote{it is allowed to be a family of twisted curves (c.f. \cite[Def.\, 4.1.2]{AV})}.
 Let $\cE$ be a locally free sheaf on $\cC$ together with an  isomorphism
  \beq\label{tauiso}\tau:\cE\mapright{\cong}  \cE\dual\otimes \omega_{\cC/D}\eeq
where $\omega_{\cC/D}$ is the dualizing sheaf.   Let
$$B:=C(\pi\lsta \cE)$$
be the direct image cone stack and $\rho:B\to D$ the natural morphism. Let $(\cC_{B},\pi_{B})$ be the universal family.
Then by Proposition \ref{deformation}
there is a relative perfect obstruction theory
$$\phi_{B/D}:\TT_{B/D}\lra \EE_{B/D}:=\rho\sta R\bul\pi\lsta(\cE).$$

  Composing $\tau$ with
  the pairing $\Vb(\cE)\times_\cC \Vb(\cE\dual\otimes\omega_{\cC/D})\mapright{(,)}\Vb(\omega_{\cC/D})$ induces canonically
 a  representable morphism of stacks
\begin{align*}
\bar h : \Vb(\cE)\lra \Vb(\omega_{\cC/D}),
\quad \bar h(x)=\frac{1}{2}(x,\tau(x)),
\end{align*}
for $x\in \Vb(\cE)$. The $\bar h$ also induces a morphism between tangent complexes
$$
d\bar h:\TT_{\Vb(\cE)/\cC}\to \bar h^*\TT_{\Vb(\omega_{\cC/D})/\cC}.
$$
 Pull back to $\cC_{B}$ via evaluation map $\ee_{B}:\cC_{B}\to\Vb(\cE)$ and apply $R^{\bullet}\pi_{B}$. We have
\begin{align*}
\begin{CD}
\sigma_{B}^\bullet:\EE_{{B}/D}\lra R^\bullet\pi_{{B}\ast}(\ee_{B}^\ast
\bar h^\ast \Omega\dual_{\Vb(\omega_{\cC/D})/\cC})\cong R^\bullet\pi_{{B}\ast}(\omega_{\cC_{B}/{B}}).
\end{CD}
\end{align*}
 It induces the following morphism:
$$\sigma_{B}:  \Ob_{B/D}=\rho\sta R^1\pi\lsta(\cE)\lra \sO_{B}.
$$

 A coordinate expression of $\sigma_{B}$ is as follows.
  For each scheme $T\to D$, let  $\cC_T=\sC\times_D T\to T$ and $\cE_T$ over $\cC_T$  be the family along $T$ via pullback.  For every $$u\in H^0(\cC_T,\cE_T),\qquad \tau(u)\in
H^0(\cC_T,\cE^{\vee}_T\otimes\omega_{\cC_T/T}),$$ and the Serre pairing
$$
H^0(\cC_T,\cE^{\vee}_T\otimes\omega_{\cC_T/T})\times H^1(\cC_T,\cE_T)\mapright{(,)} H^1(\cC_T,\omega_{\cC_{T}/T}).
$$ For arbitrary $\mathring
u\in H^1(\cC_T,\cE_T)$, one has $\tau(\mathring u)\in H^1(\cC_T,\cE^{\vee}_T\otimes\omega_{\cC_T/T})$, and
  \begin{align}\label{sigmaM}
 \sigma_B(u,\mathring u)=\frac{(\tau(u),\mathring u)+(\tau(\mathring u),u)}{2}=(\tau(u),\mathring u)\in H^1(\cC_T,\omega_{\cC_{T}/T})  \cong \Gamma(T,\sO_T),\end{align}
 where we used $(\tau(u),\mathring u)=(\tau(\mathring u),u)$ by direct checks from definition. By Serre duality and \eqref{tauiso}, the (reduced part of) degeneracy loci of $\sigma_B$ equals that of $D$. \black
\black
  The distinguished triangle
$\rho^\ast\LL_{D}\to\LL_{B}\to\LL_{B/D}\to\rho^\ast\LL_{D}[1]$
gives a morphism
$\rho^\ast\TT_{D}\to \TT_{B/D}[1]$, which
composed with $\phi_{B/D}:\TT_{B/D}\to \EE_{B/D}$ gives
$$\eta: \rho^\ast\TT_{D}\lra  \EE_{B/D}[1].
$$
Taking the cokernel of  the $H^0(\eta)$   we obtain the absolute obstruction sheaf
$$\Ob_{B}\defeq \coker\{ H^0(\eta): \rho^\ast\Omega_{D}\dual\lra  H^1(\EE_{B/D})\}.$$

\begin{coro}
The cosection $\sigma_B:\Ob_{B/D}\to \sO_{B}$ lifts to    $\bar\sigma_B: \Ob_{B}\to \sO_{B}$.
\end{coro}
\begin{proof}
 The proof is parallel to  Corollary \ref{lift}, we omit it here. \black
\end{proof}

 \black
 Denote $\bE:=h^1/h^0(\EE_{B/D})$. Suppose we have a representable morphism $R\to D$ from a DM stack $R$. Over the fiber product
 \beq\label{model}\begin{CD}
 N@>\mathfrak{b}>>B\\
 @VV{\rho_R}V@VV{\rho}V\\
 R@>>>D
 \end{CD} \eeq
 let $\bE_R:=\mathfrak{b}\sta \bE$ and let $\ti\sigma:\bE_R \to \sO_N$ be the pullback of $\sigma_{B}$. We define a cosection localized virtual pullback
$\rho_{R.\text{loc}}^!:A\lsta(R)\lra A\lsta(R)$
as the composition of \begin{align*}
\rho_{R.\text{loc}}^!:A\lsta R\mapright{\Upsilon}A\lsta\bC_{N/R}\mapright{\iota\lsta} A\lsta \bG
\mapright{0^!_{\bE_R,\mathrm{loc}}} A\lsta R,
\end{align*}
where $\Upsilon$
sends a cycle class
$\sum n_i[V_i]$ to $\sum n_i [\bC_{V_i\times_R N/V_i}]$,
$\bG:=N\times_{B} h^1/h^0(\EE_{B/D})(\sigma_{B})$, and
$0^!_{\bE_R,\mathrm{loc}}$ is the localized Gysin map defined in \cite{KL} with respect to the pullback of $\sigma_{B}$.

\begin{prop}\label{main1}For each integral closed substack $Z$ of $R$, we have
$$\rho_{R.\text{loc}}^!([Z])=(-1)^{\mu}[Z],$$
where  $\mu=min_{z\in Z}\{h^0(\cC_z,\cE|_{\cC_z})\} $ where for $z$ runs as closed points of $Z$, $\cC_z=z\times_D \cC$, and $\cE|_{\cC_z}$ denotes the pullback via $\cC_z\to \cC$.
\end{prop}
 \begin{proof}
  For  arbitrary  \'etale  $j:U\to Z$  we   take fiber products
\begin{align*}
\begin{CD}
V @>j_U>>W @>j_Z>>N\\
@V{\ff} VV@VVV@VV{\rho_R}V\\
U@>{j}>>Z@>{\subset}>>R.
\end{CD}
\end{align*}

Let $ \bE_Z:=j_Z^*\bE_R\mapright{\ti\sigma_1} \sO_{W}$
and $ \fN_U:=j_U^*j_Z^*\bE_R\mapright{\ti\sigma_2} \sO_{V}$  be the pullback of $\bE_R\mapright{\ti\sigma} \sO_{N}.$ Denote $\bG_Z:=\bG\times_{N}W$ and $\bG_U:=\bG\times_{N}V$. Then analgous to \eqref{com}, we have a commutative diagram
\beq\label{diag0}
\begin{CD}
\rho_{Z,loc}^!:A\lsta(Z)@>>>A\lsta(\bC_{W/Z})@>>>A\lsta (\bG_Z)@>0^!_{\bE_Z,loc}>>A\lsta (Z)\\
@VVj^*V @VVV @VVV@VVj^*V\\
\rho_{U,loc}^!:A\lsta(U)@>>>A\lsta(\bC_{V/U})@>>>A\lsta (\bG_U)@>0^!_{\fN_U,loc}>>A\lsta (U),
\end{CD}
\eeq

 The first square above commutes trivially. The second square commutes by \cite[Prop.\, 1.7]{Ful}. The third commutes because in  \cite[(2.1)]{KL}  the steps (i) intersection with divisors and (ii)  pushforward via
 cosection-regularizing map  are both commutative with the pullback by  \'etale  morphisms. 
  Therefore
 \begin{align*} j^*\rho^!_{Z,loc}([Z])=\rho^!_{U,loc}([U]). \end{align*}

 In Proposition \ref{main1} let us pick $U$ as follows. The morphism
 $Z\to D$ induces $\pi_Z:\cC_Z:=\cC\times_D Z\to Z$ and let $\cE_Z$ be the pull-back of $\cE$.
Then $\cE_Z$ is a locally free sheaf on $\cC$ together with, by \eqref{tauiso},
  $$\tau_Z:\cE_Z\mapright{\cong} \cE_Z\dual\otimes \omega_{\cC_Z/Z}.$$
   The function $\psi:Z\to\ZZ$ sending  $x\in Z$ to
 $$\psi(x)=\dim H^1( C_x, \cE_Z|_{C_x})$$ is upper semi-continuous. Thus we can pick an  \'etale  chart $U\to Z$ such that $\psi |_U$ is constant and is the smallest number  in $\psi (Z)$.
 This implies the $\bE_U:=R\bul\pi_{Z\ast}\cE_Z$ has cohomogies
 $H^0(\bE_U),H^1(\bE_U)$ locally free.  By shrinking $U$ smaller
  we may assume $U$ is a smooth variety and
 $$ \bE_U\cong [H^0(\bE_U)\mapright{0} H^1(\bE_U)]. $$
  Denote $j:U\to Z$ to be the inclusion.

 We see that in \eqref{diag0} $V$ is the total space of the locally free sheaf $(\pi_{Z\ast}\cE_Z)|_{U}=H^0(\bE_U)$.
 By $\tau_Z$ and Serre duality   $$ H^1(\bE_U)=\big( R^1\pi_{Z\ast}\cE_Z\big)|_U \cong (\pi_{Z\ast}\cE_Z)|_{U}\dual\cong V\dual$$
 implies $\bE_U\cong [V\mapright{0}V\dual]$, and also
$$
\fN_U=j_U^*j_Z^*\bE_R\cong \ff\sta h^1/h^0(\bE_U)=[\ff\sta V\dual/\ff\sta V]
$$ where the quotient is by zero action. Applying the composition \eqref{com} step by step,  that $
\Upsilon (U)=[V/f^*V]$ implies
$$\rho^!_{R,loc}(U)=0^!_{\fN_U,\mathrm{loc}}([V/\ff^*V])=0^!_{\ff^*V^{\vee},\mathrm{\sigma_U}}([V]),
$$
(c.f. the proof in \cite[Prop.\, 5.6]{CL1}), where $\sigma_U: H^1(\bE_U)=\ff\sta V\dual\to\sO_U$ is the pullback
of $\sigma_B$.  By \eqref{sigmaM} $\sigma_U$ is exactly the pairing with
the tautological section $\bx\in\Gamma(V,\ff\sta V)$.
To calculate the cosection localization $0^!_{\ff^*V^{\vee},\mathrm{\sigma_U}}([V])$, we follow the definition \cite[Def.\, 2.3]{KL}.
Brief $F=\ff\sta V$ and $r=\rank F$, then $\sigma_U$ is a cosection $\sigma_U:F\dual  \to \sO_V$
induced by $\sO_V\mapright{\bx} F$ . Let $R=Bl_{U}V$ and $\alpha:R\to V$ be the blow up of $V$ along the zero section $U$ of $V\to U$. Denote $E=\alpha^{-1}(U)$ as the exceptional divisor.
Then the zero loci of $\alpha\sta \bx\in \Gamma(R,\alpha\sta\ff\sta V)$ is $E$. Denote $\alpha(\sigma):E\cong \PP(V)\to V$ as the restriction of $\alpha$.
The exact sequence $F\dual\mapright{\sigma_U} \sO_V\to\sO_U\to 0$ pullbacks to the exact
$$\alpha\sta F\dual \mapright{\alpha\sta\sigma_U} \sO_R\to\sO_E\to 0.$$
Thus $\alpha\sta\sigma_U$ factors through a surjective $\ti\sigma: \ti F\dual:=\alpha\sta F\dual\to \sO_R(-E)$. Let $\ti G=ker\ \ti\sigma$. Then
\begin{align*} \ti G\dual|_E\cong coker\{\sO_R(E)|_E\to \alpha(\sigma)\sta F\}\cong T_{E/U}\otimes_E\sO(-1) \end{align*}
 and  by the construction in  \cite[(2.1)]{KL}
\begin{align*}
0^!_{\ff^*V^{\vee},\mathrm{\sigma_U}}([V])&=\alpha(\sigma)\lsta ([-E] \cap 0^!_{\ti G}[R])=(-1)\alpha(\sigma)\lsta \bigg( e(\ti G|_E) \cap [E] \bigg)  \\
&= (-1)(-1)^{r-1} \alpha(\sigma)\lsta e(T_{E/U}\otimes_E \sO(-1))=(-1)^r [U].
\end{align*}
 Therefore \beq\label{final} \rho^!_{R,loc}(U)=(-1)^r [U].\eeq
 As  $Z$ is integral
 by dimension reason  $\rho^!_{Z,loc}(Z)=c[Z]$ for some $c\in \QQ$, where
  \eqref{diag0} and \eqref{final} forces $c=(-1)^r$. This proves Proposition \ref{main1}.
\end{proof}

\black

 \subsection{}
 
  We are ready to prove Theorem \ref{theo1} from the theorem \ref{main1}.

\begin{proof}[Proof of Theorem \ref{theo1}] We claim
$$[\cpgp]\virt=[\cN]\virt=(-1)^{md+1-g}[\cqg]^{vir}.$$

The first equality is from Corollary \ref{shriek}. For the second identity, compare \eqref{v-v} with
\eqref{model} by letting $D=\cD,B=\cB, R=\cqg$,
$\cE=\sL^{\otimes m}_{\cD}\oplus\sR_{\cD} $,
and
$\tau:\cE\to  \omega_{\cC_{\cD}/\cD}\otimes \cE\dual $ is determined by $\sL^{\otimes m}_{\cD}{\cong}\omega_{\cC_{\cD}/\cD}\otimes \cP_{\cD}^{-1}$,  $\cP_{\cD}{\cong}\omega_{\cC_{\cD}/\cD}\otimes \cL_{\cD}^{-m}$. Under such setup  $N=\cN$ and $\bE_R\cong h^1/h^0(\EE_{\cN/\cQ})$.

 The $\sigma_B$ defined in \eqref{sigmaM} admits the following description.
  For each closed point $\xi\in B$   represented by
$$(u_1,u_2)\in   H^0(C,L^{\otimes m})\times H^0(C,L^{\vee\otimes m}\otimes \omega_C)
$$
where $(C,p_1,\cdots,p_\ell ,L)\in \cD$ is the point under $\xi$,
for arbitrary $u_1\in H^1(C,L^{\otimes m})$ and $u_2\in
H^1(C,L^{\vee\otimes m}\otimes \omega_C)$, the cosection is
$\sigma_B|_{\xi}(\mathring u_1,\mathring u_2)=u_1\cdot\mathring
u_2+u_2\cdot\mathring u_1 $. A family version of above implies $\gamma\sta\sigma_B$
coincides with the $\sigma_0$ defined below \eqref{si-V} (c.f.Section
4.3). \black
 Denote $[\cQ]^{vir}=\sum_im_i[\cQ_i]$, where $\cQ_i$ is an integral substack. By Lemma \ref{pullback} and Theorem \ref{main1}
$$
[\cN]\virt=v^!_{loc}([\cqg]^{vir})=\sum_i(-1)^{\mu_i}m_i[\cQ_i]=(-1)^{md+1-g}[\cqg]^{vir},
$$
where $\mu_i=h^0(C,\cE|_C)=\chi(C,L^{\otimes m})=md+1-g$.
\end{proof}

\section{Appendix}

 All the arguments in the paper can be extended to the complete intersection in $\PP^n$ cases.

 Let $X\subset\PP^n$ be a smooth complete intersection defined by equations $q_1\cdots,q_r$ with $\deg q_i=m_i$. Let $M:=\qme(\PP^n,d)$ be the moduli of genus $g$ degree $d$ $\epsilon$-stable quasimaps to $\PP^n$. We denote by $(\mathcal{C}_M, \pi_M)$ be the universal family of $\qme(\PP^n,d)$, and $\mathscr{L}_{M}$ be the university line bundle. Denote
$$
\mathscr{P}^i_{M}:=\mathscr{L}_{M}^{\vee\otimes m_i}\otimes\omega_{\mathcal{C}_M/M},\quad i=1,\cdots, r.
$$

 Using direct image cone construction, we form the DM stack $$\overline{\mathcal{P}}:=C(\oplus_i\pi_{M*}\mathscr{P}^i_{M}).$$  
 Letting $\pi_{\overline{\mathcal{P}}}:\cC_{\overline{\mathcal{P}}}\to \overline{\mathcal{P}}$ and $\sL_{\overline{\mathcal{P}}}$ be the universal curve and line bundle.   The obstruct theory of $\overline{\mathcal{P}}$ is

 $$\phi_{\overline{\mathcal{P}}/\cD}:\TT_{\overline{\mathcal{P}}/\cD}\lra \EE_{\overline{\mathcal{P}}/\cD}\defeq
R^\bullet\pi_{\overline{\mathcal{P}} \ast} (\sL^{\oplus (n+1)}_{\overline{\mathcal{P}}}\bigoplus\oplus_i\sP^i_{\overline{\mathcal{P}}}).
$$

We define a multi-linear bundle morphism from $q_i$
\begin{equation}\label{co2}
h_1: \Vb(\sL^{\oplus (n+1)}_{\cD} \bigoplus\oplus_i\sP^i_{\cD})\lra \Vb(\omega_{\cC_{\cD}/\cD}),
\quad h_1(x,p)=\sum_ip_i\cdot q_i(x),
\end{equation}
 It induces the following cosection:
$$\overline{\sigma}_m:  \Ob_{\overline{\mathcal{P}}/\cD}=R^1\pi_{\overline{\mathcal{P}} \ast}\sL^{\oplus (n+1)}_{\overline{\mathcal{P}}} \bigoplus\oplus_i R^1\pi_{\overline{\mathcal{P}}\ast} \sP^i_{\overline{\mathcal{P}}}\lra \sO_{\overline{\mathcal{P}}}.
$$
the degeneracy loci of $\overline{\sigma}_m$ is $\qme(X,d)$. The same argument as Theorem \ref{theo1} implies
\begin{theo}\label{theo2}
For $g\ge0$, $\epsilon>0$ and $\ell$ be a nonnegative integer,
$$
[\overline{\mathcal{P}}]\virt_{\mathrm{loc}}=(-1)^{(\sum_i m_i) d+1-g}[\qme(X,d)]\virt\in A_{*} \qme(X,d).
$$
\end{theo}


\begin{thebibliography}{99}

\bibitem[AV]{AV} D. Abramovich and A. Vistoli, {\em Compactifying the space of stable maps},
Journal of the American Mathematical Society. {\bf 15} (2001), no. 1, 27-75.






\bibitem[CJR]{CJR}  E. Clader, F. Janda, Y.B. Ruan, \emph{Higher-genus wall-crossing in the gauged linear sigma model}, arXiv:1706.05038.



\bibitem[CL]{CL} H.L. Chang and J. Li, {\em Gromov-Witten Invariants of Stable Maps with Fields}, International Mathematics Research Notices.  {\bf 18} (2012), 4163-4217.

\bibitem[CL1]{CL1} H.L. Chang and J. Li, {\em An algebraic proof of the hyperplane property of the genus one GW-invariants of quintics}, J. Diff. Geom.  {\bf 100} (2015), 251-299.

\bibitem[CKL]{CKL} H.L. Chang, Y.H. Kiem and J. Li, {\em Torus localization and wall crossing for cosection localized virtual cycles},  Adv. Math. {\bf 308} (2017), 964-986.

\bibitem[CLLL]{CLLL}   H.L. Chang, J. Li, W.P. Li, C.C. Melissa Liu,
 \emph{Mixed-Spin-P fields of Fermat quintic polynomials}, arXiv:1505.07532.

\bibitem[CLLL1]{CLLL1}   H.L. Chang, J. Li, W.P. Li, C.C. Melissa Liu,
 \emph{An effective theory of GW and FJRW invariants of quintics Calabi-Yau manifolds}, arXiv:1603.06184.

\bibitem[FJR]{FJR} H.J. Fan, T. Jarvis, Y.B. Ruan, \emph{A Mathematical Theory of the Gauged Linear Sigma Model}, Geom.Topol. {\bf 22} (2018) 235-303.

\bibitem[Ful]{Ful} W. Fulton, {\em Intersection theory}, Springer-Verlag, New York, 1984.

\bibitem[FK0]{FK0} Ionut Ciocan-Fontanine, Bumsig Kim, \emph{Moduli stacks of stable toric quasimaps}, Advances in Mathematics 225 (2010), 3022-3051.

\bibitem[FK]{FK} I. Ciocan-Fontanine, B. Kim, \emph{Wall-crossing in genus zero quasimap theory and mirror maps}, Algebraic Geometry. {\bf 4} (2014), 400-448.

\bibitem[FK1]{FK1} I. Ciocan-Fontanine, B. Kim, \emph{Quasimap Wall-crossings and Mirror Symmetry}, arXiv:1611.05023.

\bibitem[FKM]{FKM} I. Ciocan-Fontanine, B. Kim, D. Maulik, \emph{Stable quasimaps to GIT quotients}, Journal of Geometry and Physics, {\bf 75} (2014), 17-47.



\bibitem[Illusie]{Illusie} L. Illusie, \emph{Complexe cotangent et deformations I,II},  Lecture Notes in Mathematics Nos. 239,283. Springer, Berlin, Heidelberg, New York, 1971.



\bibitem[KKP]{KKP} B. Kim, A. Kresch and T. Pantev, \emph{Functoriality in intersection theory and a conjecture of Cox, Katz, and Lee}, J. Pure Appl. Algebra. {\bf 179} (2003), no. 1-2, 127-136.

\bibitem[K-L]{K-L} B. Kim, H. Lho, \emph{Mirror Theorem for Elliptic Quasimap Invariants}, arXiv:1506.03196.


\bibitem[KL]{KL} Y.H. Kiem and J. Li,  \emph{Localized virtual cycle by cosections},
J. Amer. Math. Soc. {\bf 26} (2013), no. 4, 1025-1050.



\bibitem[LZ]{LZ} J. Li and A. Zinger, \emph{On the Genus-One Gromov-Witten Invariants of Complete Intersections}, J. of Differential Geom. {\bf 82}  (2009), no. 3, 641-690.

\bibitem[MOP]{MOP} A. Marian, D. Oprea and R. Pandharipande \emph{The moduli space of stable quotients}, Geom. Topol. \textbf{15}(2011), no. 3, 1651-1706.

\bibitem[Zi]{Zi} A. Zinger, \emph{The reduced genus 1 Gromov-Witten invariants of Calabi-Yau hypersurfaces},  J. Amer. Math. Soc. {\bf 22}  (2009),  no. 3, 691-737.


\end{thebibliography}
\end{document}